\documentclass[11pt]{amsart}

\usepackage{amsmath,amssymb,amsfonts,bm,latexsym,yfonts,cite,mathtools}
\usepackage{geometry}

\newtheorem{definition}{Definition}[section]
\newtheorem{proposition}[definition]{Proposition}
\newtheorem{theorem}[definition]{Theorem}
\newtheorem{lemma}[definition]{Lemma}

\newtheorem{corollary}[definition]{Corollary}

\numberwithin{equation}{section}

\newcommand{\C}{\mathbb{C}}

\newcommand{\N}{\mathbb{N}}
\newcommand{\R}{\mathbb{R}}
\renewcommand{\S}{\mathbb{S}}

\newcommand{\cG}{\mathcal{G}}
\newcommand{\cM}{\mathcal{M}}

\newcommand{\cS}{\mathcal{S}}

\title[Estimates for the $k$-plane transform]{Mapping estimates for the $k$-plane transform in Sobolev, Besov, and Triebel--Lizorkin Spaces
}

\author{Fatma Terzioglu}
\address{North Carolina State University, Department of Mathematics, Raleigh, NC 27695, USA.}
\email{fterzioglu@ncsu.edu}


\begin{document}

\begin{abstract}
We study mapping properties of the $k$-plane transform in Sobolev, Besov, and Triebel--Lizorkin spaces. For $1\le k\le d-1$, the $k$-plane transform integrates a function over $k$-dimensional affine planes in $\R^d$, yielding a function on the affine Grassmannian $\cG_{k,d}$. First, we establish Sobolev stability estimates for compactly supported functions, extending classical results of Natterer for the X-ray ($k=1$) and Radon ($k=d-1$) transforms to the general $k$-plane transform. Second, we extend isometry identities for the Radon and X-ray transforms, due to Reshetnyak, Sharafutdinov, and Kindermann--Hubmer, to the $k$-plane transform. Finally, we prove boundedness of the $k$-plane transform in Besov and Triebel--Lizorkin spaces. 
\end{abstract}

\maketitle

 \tableofcontents

\section{Introduction}
In this paper we study mapping properties of the $k$-plane transform in Sobolev, Besov, and Triebel--Lizorkin spaces. For $d\ge2$ and $1\le k\le d-1$, the $k$-plane transform integrates a function over $k$--dimensional affine planes in $\R^d$. The collection of all such planes is the affine Grassmannian $\cG_{k,d}$. 

The case $k=1$ corresponds to the X-ray (or John's) transform, which integrates functions over lines in $\R^d$ \cite{John1994}, while the case $k=d-1$ yields the classical Radon transform, which integrates over affine hyperplanes \cite{Natterer,Helgason}. These transforms arise in several areas of mathematics (see, e.g., \cite{Markoe2006,Gardner2006,Rubin2015}) and play a central role in imaging techniques such as Computed Tomography and Magnetic Resonance Imaging, where their inversion enables reconstruction of images from projection data \cite{Natterer,NattWubb}.

The analytic properties of the $k$-plane transform have been widely studied. Inversion formulas in various function spaces were developed in \cite{Keinert1989,Rubin1998,Rubin2004,Rubin2012,Parhi2024}. Range characterizations were obtained in \cite{Gonzalez1990,Gonzalez1991,Kurusa1991,Helgason}, while microlocal properties were analyzed in \cite{Chihara2022}. 

Our focus is on \emph{mapping estimates} for the $k$-plane transform $P$, that is, bounds of the form
\[
\|Pf\|_{X(\cG_{k,d})} \le C\,\|f\|_{Y(\R^d)},
\]
where $X(\cG_{k,d})$ and $Y(\R^d)$ are function spaces on the affine Grassmannian $\cG_{k,d}$ and on $\R^d$, respectively. Such inequalities quantify how the regularity or integrability of a function is affected by integration over $k$-dimensional planes. Mapping properties of this type are well studied in the case of Lebesgue spaces \cite{Christ1984,Drury1983,Drury1984,Drury1989,Drouot2014,Sato2022,SmithSolmon1975,Solmon1979,Strichartz1981,Calderon1983,OberlinStein1982,Duoandikoetxea2002,Duoandikoetxea2001,Rubin2014,Quinto1982,Kumar2010,Flock2025}.

Sobolev space estimates, which additionally measure regularity, have been studied only for the X-ray and Radon transforms. Two-sided estimates for compactly supported functions in Sobolev spaces were obtained by Natterer \cite[Theorem~5.1]{Natterer}. In the case of the Radon transform ($k=d-1$), Reshetnyak \cite{GGV1966} established an isometry identity in the $L^2$ setting. Sharafutdinov later extended this result to Sobolev and weighted Sobolev spaces \cite{Sharafutdinov2021radon,Sharafutdinov2017tensor}, and also treated the X-ray transform under similar assumptions \cite{Sharafutdinov2021x-ray}. Recently, Kindermann and Hubmer \cite{Kindermann2025} extended the Radon transform isometry to weighted fractional Sobolev spaces and used it to derive nonlinear filtered backprojection algorithms.

Beyond their analytic significance, these estimates also have geometric consequences. When applied to indicator functions of convex bodies they yield inequalities controlling the volumes of their sections in higher codimension (see e.g., \cite{Gardner2006,Rubin2019}). In terms of applications, such bounds are closely related to stability properties of inversion algorithms. For further perspectives on the affine Grassmannian in mathematics and applied statistics we refer to \cite{Lim2021}.

The goal of this paper is to extend these developments in several directions. Our main contributions are the following:
\begin{itemize}
\item In Theorem~\ref{thm:sobolev-estimate}, we establish two-sided Sobolev estimates for the $k$-plane transform for compactly supported functions, extending Natterer's results for the X-ray and Radon transforms to all $1 \le k \le d-1$.
\item In Theorem~\ref{thm:P-isometry}, we extend the isometry identities known for the Radon and X-ray transforms to the $k$-plane transform.
\item In Theorem~\ref{thm:BL-TL-estimate}, we prove boundedness of the $k$-plane transform in Besov and Triebel--Lizorkin spaces.
\end{itemize}

Besov and Triebel–Lizorkin spaces provide a finer description of regularity by decomposing functions into components localized in different frequency bands via Littlewood–Paley decompositions. They cover many classical function spaces such as Sobolev spaces, fractional Sobolev (Bessel--potential) spaces, Hölder--Zygmund spaces, local Hardy spaces, and the space of functions of bounded mean oscillation \cite{Triebel1983,Triebel1992,Grafakos2014classical,Grafakos2014modern,Hytonen2023}. Besides their theoretical interest, the estimates obtained here may also contribute to the development of new numerical methods for inversion of the $k$-plane transform.

The paper is organized as follows. In Section~\ref{sec:preliminaries}, we introduce the $k$-plane transform, review its basic properties, and establish two additional properties (Propositions~\ref{prop:intertwine-kplane} and \ref{prop:vector-valued inequality}) needed for the later results. Section~\ref{sec:Lp-estimates} summarizes known Lebesgue space estimates for the $k$-plane transform, which will be instrumental in our subsequent analysis in Besov and Triebel–Lizorkin spaces. In Section~\ref{sec:Sobolev-estimates}, we develop Sobolev space estimates, including weighted Sobolev estimates. Section~\ref{sec:BL-TL-estimates} introduces Besov and Triebel–Lizorkin spaces on the affine Grassmannian and establishes mapping estimates for the $k$-plane transform in these spaces. Finally, Section~\ref{sec:conclusion} contains concluding remarks.

\section{$k$-plane Transform and its Properties}\label{sec:preliminaries}

Let $d\ge2$ and $1\le k\le d-1$. 
We denote by $G_{k,d}$ the Grassmannian of all $k$--dimensional subspaces of $\mathbb R^d$. 
For $\alpha\in G_{k,d}$, let $\alpha^\perp$ denote its orthogonal complement, which has dimension $d-k$.

The affine Grassmannian $\mathcal G_{k,d}$ is the space of all affine $k$--planes in $\mathbb R^d$, parametrized by
\[
\mathcal G_{k,d}
=
\{(\alpha,y): \alpha\in G_{k,d},\ y\in\alpha^\perp\},
\]
where the pair $(\alpha,y)$ represents the affine plane
\[
\alpha+y=\{x+y : x\in\alpha\}.
\]

\medskip

We denote by $\mathcal S(\mathbb R^d)$ the Schwartz space on $\mathbb R^d$ and by $\mathcal S'(\mathbb R^d)$ its dual.

The Schwartz space $\mathcal S(\mathcal G_{k,d})$ consists of all functions 
$u\in C^\infty(\mathcal G_{k,d})$ such that for every differential operator 
$D_\alpha$ on $G_{k,d}$, every constant--coefficient differential operator 
$D_y$ in the fiber variable $y$, and every $N\ge0$,
\[
\sup_{(\alpha,y)\in\mathcal G_{k,d}}
(1+|y|)^N
\bigl|
(D_\alpha D_y u)(\alpha,y)
\bigr|
<\infty.
\]
Equivalently, one may take $D_y=\partial_y^\beta$ for multi-indices $\beta$.
In local coordinates given by a smooth orthonormal frame
$\gamma(\alpha):\mathbb R^{d-k}\to\alpha^\perp$ (so $y=\gamma(\alpha)x$),
this is equivalent to uniform Schwartz estimates in $x\in\mathbb R^{d-k}$,
uniformly in $\alpha$; see \cite{Gonzalez1991}.

The space $\mathcal S'(\mathcal G_{k,d})$ of tempered distributions on
$\mathcal G_{k,d}$ is the continuous dual of $\mathcal S(\mathcal G_{k,d})$:
\[
\mathcal S'(\mathcal G_{k,d})
=
\{\,T:\mathcal S(\mathcal G_{k,d})\to\mathbb C
\text{ linear and continuous}\,\},
\]
and the action is denoted by $\langle T,u\rangle$.

\medskip

Throughout the paper we use the unitary Fourier transform. 
For $f\in\mathcal S(\mathbb R^d)$,
\[
\hat f(\xi)
=
(2\pi)^{-d/2}\int_{\mathbb R^d} e^{-ix\cdot\xi} f(x)\,dx,
\qquad
f(x)
=
(2\pi)^{-d/2}\int_{\mathbb R^d} e^{ix\cdot\xi}\hat f(\xi)\,d\xi.
\]
With this normalization, Plancherel's theorem states that
\begin{equation}\label{Plancherel}
\int_{\mathbb R^d} f(x)\,\overline{g(x)}\,dx
=
\int_{\mathbb R^d} \hat f(\xi)\,\overline{\hat g(\xi)}\,d\xi, \qquad f,g\in\mathcal S(\mathbb R^d).
\end{equation}

For $u\in\mathcal S(\mathcal G_{k,d})$ and fixed $\alpha\in G_{k,d}$,
we define the fiberwise Fourier transform in the variable
$y\in\alpha^\perp$ by
\[
\hat u(\alpha,\eta)
=
(2\pi)^{-(d-k)/2}
\int_{\alpha^\perp} e^{-iy\cdot\eta} u(\alpha,y)\,dy,
\qquad
\eta\in\alpha^\perp.
\]
This definition is independent of the choice of orthonormal coordinates on
$\alpha^\perp$ and extends uniquely to a continuous automorphism of
$\mathcal S'(\mathcal G_{k,d})$ acting fiberwise via
\[
\langle \widehat T,u\rangle=\langle T,\hat u\rangle.
\]

\medskip

With these preparations, we introduce the $k$--plane transform.

\begin{definition}[$k$--plane transform]
For $f\in\mathcal S(\mathbb R^d)$, the $k$--plane transform $P$ is defined by
\[
(Pf)(\alpha,y)
=
\int_{x\in\alpha} f(x+y)\,dx,
\qquad
\alpha\in G_{k,d},\ y\in\alpha^\perp,
\]
where $dx$ denotes Lebesgue measure on $\alpha$.
\end{definition}

The mapping $P:\mathcal S(\mathbb R^d)\to\mathcal S(\mathcal G_{k,d})$ is continuous \cite{Gonzalez1991}.

\medskip

The $k$--plane transform admits a particularly simple description in the
Fourier domain. The following Fourier--slice identity is classical and
fundamental for the analysis of $P$; we include the proof for completeness.

\begin{theorem}[Fourier--slice theorem]\label{t:FST}
For every $f\in\mathcal S(\mathbb R^d)$,
\[
\widehat{Pf}(\alpha,\eta)
=
(2\pi)^{k/2}\hat f(\eta),
\qquad
\alpha\in G_{k,d},\ \eta\in\alpha^\perp,
\]
where $\widehat{Pf}(\alpha,\eta)$ denotes the fiberwise Fourier transform
of $Pf(\alpha,\cdot)$ in $y$, and $\hat f$ is the Fourier transform of
$f$ restricted to $\alpha^\perp$.
\end{theorem}

\begin{proof}
Fix $\alpha\in G_{k,d}$ and $\eta\in\alpha^\perp$.
Using Fubini's theorem and the orthogonal decomposition
$\mathbb R^d=\alpha\oplus\alpha^\perp$, we compute
\begin{align*}
\widehat{Pf}(\alpha,\eta)
&=(2\pi)^{-(d-k)/2}
\int_{\alpha^\perp} e^{-iy\cdot\eta}
\Bigl(\int_\alpha f(x+y)\,dx\Bigr)\,dy \\
&=(2\pi)^{-(d-k)/2}
\int_{\mathbb R^d} e^{-iz\cdot\eta} f(z)\,dz
= (2\pi)^{k/2}\hat f(\eta),
\end{align*}
since $x\cdot\eta=0$ for $x\in\alpha$.
\end{proof}

\medskip

The next property follows directly from Theorem~\ref{t:FST} and will be
used later in the study of mapping properties in Section~\ref{sec:BL-TL-estimates}.

\begin{proposition}[Intertwining property]\label{prop:intertwine-kplane}
Let $\varphi:\mathbb R^d\to\mathbb C$ be bounded and Borel measurable.
Define the Fourier multiplier
\[
(Mf)\widehat{\ }(\xi)=\varphi(\xi)\hat f(\xi),
\]
and define the fiberwise multiplier $\mathcal M$ on $\mathcal S(\mathcal G_{k,d})$ by
\[
(\mathcal M u)\widehat{\ }(\alpha,\eta)
=
\varphi(\eta)\hat u(\alpha,\eta).
\]
Then for every $f\in\mathcal S(\mathbb R^d)$,
\[
\mathcal M(Pf)=P(Mf).
\]
\end{proposition}

\begin{proof}
The identity follows immediately from Theorem~\ref{t:FST}
and injectivity of the fiberwise Fourier transform.
\end{proof}

\medskip

Next we present two auxiliary results that will be used later.

We recall that the quasi-Banach space $\ell^r$, $0 < r \le \infty$, contains all sequences
$\{a_j\}_{j \in \mathbb{N}_0}$ of complex numbers such that the quasi-norm
\[
\| \{a_j\} \|_{\ell^r} :=
\begin{cases}
\left( \displaystyle \sum_{j=0}^{\infty} |a_j|^r \right)^{1/r},
& 0 < r < \infty, \\[1.2ex]
\displaystyle\sup_{j \in \mathbb{N}_0} |a_j|,
&  r = \infty,
\end{cases}
\]
is finite. It is a Banach space if $1 \le r \le \infty$.

\begin{proposition}[Vector-valued inequality for the $k$-plane transform]\label{prop:vector-valued inequality}
Let $1\le r\le \infty$. Then, for any $(\alpha,y)\in\mathcal G_{k,d}$,
\begin{align}\label{eq:vector-valued inequality}
\|\{P(h_j)(\alpha,y)\}\|_{\ell^r} = \Big(\sum_{j=0}^\infty \big|P(h_j)(\alpha,y)\big|^r\Big)^{1/r} \!\!
\le
P\Big(\big(\sum_{j=0}^\infty |h_j|^r\big)^{1/r}\Big)(\alpha,y) = P (\|\{h_j\}\|_{\ell^r})(\alpha,y),
\end{align}
for every sequence $\{h_j\}$ of measurable functions for which the right-hand side of \eqref{eq:vector-valued inequality} is finite
(with sums replaced by a supremum when $r=\infty$).
\end{proposition}

\begin{proof}
Let $(\alpha,y)\in\mathcal G_{k,d}$. By the triangle inequality for integrals,
\begin{align}\label{eq:tri_ineq_int}
    |P(h_j)(\alpha,y)| = \left|\int_\alpha h_j(x+y)dx \right| \leq \int_\alpha |h_j(x+y)|dx = P(|h_j|)(\alpha,y).
\end{align}

Assume first $1\le r<\infty$. Taking the $\ell^r$-norm in $j$ in \eqref{eq:tri_ineq_int} and applying Minkowski’s integral inequality (viewing the sum as integration with respect to counting measure on $\N_0$; see e.g., \cite{Folland}, Theorem 6.19), we obtain
\begin{align*}
&\|\{P(h_j)\}(\alpha,y)\|_{\ell^r}
\le \Big(\sum_{j=0}^\infty (P(|h_j|)(\alpha,y))^r\Big)^{1/r}
= \Big(\sum_{j=0}^\infty \Big(\int_\alpha |h_j(x+y)|\,dx\Big)^r\Big)^{1/r} \\
&\le \int_\alpha \Big(\sum_{j=0}^\infty |h_j(x+y)|^r\Big)^{1/r} dx 
 = P \Big(
\Big(\sum_{j=0}^\infty |h_j|^r\Big)^{1/r}
\Big)(\alpha,y)
= P\!\left(\|\{h_j\}\|_{\ell^r}\right)(\alpha,y).
\end{align*}

For $r=\infty$, taking the supremum over $j$ in \eqref{eq:tri_ineq_int} gives
\begin{align*}
\|\{P(h_j)\}(\alpha,y)\|_{\ell^\infty}
&= \sup_{j\in\N_0}|P(h_j)(\alpha,y)|
\le \sup_{j\in\N_0} \int_\alpha |h_j(x+y)|\,dx\\
& \le \int_\alpha \sup_{j\in\N_0}|h_j(x+y)|\,dx 
= P(\|\{h_j\}\|_{\ell^\infty})(\alpha,y),
\end{align*}
which completes the proof.
\end{proof}

\begin{lemma}[\cite{Markoe2006}, Theorem 3.9] \label{generalized_polar}
If $f$ is an integrable function on $\R^d$, then
\begin{align}
    \int_{G_{k,d}} \int_{\alpha^\perp} f(y) |y|^k dy d\alpha = |G_{k,d-1}|\int_{\R^d} f(x)dx,
\end{align}
where the total measure of the Grassmannian is defined as
\begin{align*}
    |G_{k,d}| 
    &= \frac{|\S^{d-1}||\S^{d-2}| \cdots |\S^{d-k}|}{2|\S^{k-1}||\S^{k-2}| \cdots |\S^1|}, \quad k \ge 2,\\
    |G_{1,d}| &= \frac{|\S^{d-1}|}{2}\\
    |G_{0,d}| &= 1.    
\end{align*}
\end{lemma}

\section{Lebesgue Space Estimates}\label{sec:Lp-estimates}
In this section, we review known $L^p$ mapping properties of the
$k$--plane transform that guarantee the integrability of functions
over $k$--dimensional affine planes. These results will play a
central role in our later analysis in Besov and
Triebel--Lizorkin spaces in section \ref{sec:BL-TL-estimates}.

For $0 < p \le \infty$, the space $L^p(\R^d)$ consists of all measurable functions $f$ on $\R^d$ such that the quasi-norm
\[
\| f \|_{L^p} :=
\begin{cases}
\left( \displaystyle\int_{\mathbb{R}^d} |f(x)|^p \, dx \right)^{1/p},
& 0 < p < \infty, \\
\operatorname*{ess\,sup}_{x \in \mathbb{R}^d} |f(x)|,
& p = \infty,
\end{cases}
\]
is finite.

Since the $k$--plane transform depends on both the orientation of the plane and its translation, it is natural to consider mixed Lebesgue norms on the affine Grassmannian. 
For $1\le q,t\le\infty$, we define the mixed-norm space $$L^q_\alpha(L^t_y):=L^q(G_{k,d};L^t(\alpha^\perp)),$$ 
as the set of measurable functions $u$ on $\cG_{k,d}$ such that
\begin{align}\label{mixed-norm}
\|u\|_{L^q_\alpha(L^t_y)}
:=
\left(\int_{G_{k,d}}
\left(\int_{\alpha^\perp}|u(\alpha,y)|^t\,dy\right)^{q/t}
\,d\alpha\right)^{1/q},
\end{align}
is finite.

If $t=q$, we simply write $L^q(\cG_{k,d})$ and the norm becomes
\begin{align}\label{single-norm}
\|u\|_{L^q(\cG_{k,d})}
:=
\left(\int_{G_{k,d}}
\int_{\alpha^\perp}|u(\alpha,y)|^q\,dy
\,d\alpha\right)^{1/q}.
\end{align}

\medskip

For integers $d\ge2$ and $1\le k\le d-1$, we are interested in determining the exponents $1\le p,q,t\le\infty$ for which the estimate
\begin{align}\label{eq:mixed-estimate}
\|Pf\|_{L^q_\alpha(L^t_y)}\le C\,\|f\|_{L^p(\R^d)}
\end{align}
holds with a constant $C>0$ depending only on $d,k,p,q,t$.

We remark that \eqref{eq:mixed-estimate} can hold only if
\begin{align}\label{eq:p-cond}
1 \leq p < \frac{d}{k}.
\end{align}
Indeed, for $p\ge d/k$ one may take
$$f(x) = (1+|x|)^{-k/d}(\log(3+|x|))^{-\delta},$$
with $k/d<\delta<1$, so that $f\in L^p(\R^d)$ while $Pf(\alpha,y)$ diverges; cf.\ \cite{Christ1984}.

Moreover, by considering $f= \chi_{B}$, the characteristic function of a ball or a box $B$, one can show that the following two conditions are also necessary \cite{Drury1989,Sato2022}:
\begin{align}\label{eq:qt-cond}
\frac{d}{p} - \frac{d-k}{t} = k, \qquad q \le (d-k)p', \qquad p'=\frac{p}{p-1}.
\end{align}
We also note that the estimate 
\begin{align}\label{L1-L1-estimate}
\|Pf\|_{L_{\alpha}^q(L_y^1)}\le C_{d,k}\,\|f\|_{L^1(\R^d)}, \quad 1 \leq q \leq \infty,
\end{align}
follows directly from Fubini's theorem. This implies that every integrable function $f$ on $\R^d$ is integrable over almost every affine $k$--plane in $\R^d$ \cite{SmithSolmon1975}.

The first nontrivial Lebesgue estimates for the $k$--plane transform were obtained by Smith and Solmon \cite{SmithSolmon1975}. 
They showed that when $k<d/2$, the operator $P$ is bounded from $L^2(\R^d)$ to $L_\alpha^2(L_y^t)$ provided
\[
\frac{d}{2} - \frac{d-k}{t} = k .
\]

Solmon \cite{Solmon1979} later extended this result to the range
\[
1 \le p \le 2,
\qquad 
p<\frac{d}{k},
\qquad 
\frac{d}{p} - \frac{d-k}{t} = k,
\qquad q=2.
\]

Strichartz \cite{Strichartz1981} subsequently proved an analogue of Solmon's result with $q=p'=\frac{p}{p-1}$ (excluding $p=1$; see also \cite{Calderon1983}), and also established the estimate \eqref{eq:mixed-estimate} for
\[
1 < p <\frac{d}{k},
\qquad 
\frac{d}{p} - \frac{d-k}{t} = k,
\qquad q=1.
\]

\medskip

In the case $k=d-1$, corresponding to the classical Radon transform, Oberlin and Stein \cite{OberlinStein1982} proved that the conditions \eqref{eq:p-cond} and \eqref{eq:qt-cond} are both necessary and sufficient for \eqref{eq:mixed-estimate} to hold.

For $k=1$, corresponding to the x-ray transform, Drury \cite{Drury1983} obtained the estimate \eqref{eq:mixed-estimate} for
\[
1 \le p < \frac{d+1}{2},
\qquad 
t = q = \frac{(d-1)p}{d-p}.
\]
He later extended this analysis to higher dimensions, showing that when $k \ge (d-1)/2$ the estimate holds for
\[
1 \le p \le \frac{d+1}{k+1},
\qquad 
t = q = \frac{(d-k)p}{d-kp}
\]
\cite{Drury1984} (see also \cite{Drury1989,Duoandikoetxea2002,Sato2022}). 
When $d=2$, the Oberlin--Stein theorem gives the complete result for $k=1$.

\medskip

For $1 \le k \le d-2$, the most general result to date was obtained by Christ \cite{Christ1984} (see also \cite{Drury1989}), which is as follows.

\begin{theorem}[\cite{Christ1984}]\label{thm:christ-mixed}
Let $d\ge 2$ and $1\le k \le d-2$. 
Assume either
\begin{align}\label{christ-ThmA}
1 \le p \le \frac{d+1}{k+1},
\end{align}
or 
\begin{align}\label{christ-ThmB}
1 \le p\le 2\qquad\text{and}\qquad p<\frac{d}{k}.
\end{align}
Then for any $q,t \ge 1$ satisfying \eqref{eq:qt-cond}, there exists a $C>0$ (depending on $k,d,p,q,t$) such that
\begin{align}\label{eq:christ-mixed-estimate}
\|Pf\|_{L^q_\alpha(L^t_y)}\le C\,\|f\|_{L^p(\R^d)}.
\end{align}
\end{theorem}

The relative strength of the conditions \eqref{christ-ThmA} and \eqref{christ-ThmB} depends on the dimension $k$. 
If $k<\tfrac{d-1}{2}$, then $2 < \tfrac{d+1}{k+1} < \tfrac{d}{k}$ and condition \eqref{christ-ThmA} provides the larger admissible range of $p$. 
If $k \ge \tfrac{d-1}{2}$, then $\tfrac{d+1}{k+1}\le2$, and the admissible range of $p$ is determined by \eqref{christ-ThmB}. More specifically, in this case $\tfrac{d}{k}>2$ if and only if $d=2k+1$, which holds only when $d$ is odd, and thus the theorem holds for $1 \leq p \leq 2$. Otherwise, $ \tfrac{d+1}{k+1} < \tfrac{d}{k} \leq 2$, so the estimate holds for $1 \leq p \leq \tfrac{d}{k}$. Christ conjectured that the necessary conditions \eqref{eq:p-cond}--\eqref{eq:qt-cond} might also be sufficient in the general case $1\le k\le(d-1)/2$, but this problem remains open except for radial functions \cite{Duoandikoetxea2001}.

\medskip

Mapping properties of the $k$--plane transform in weighted Lebesgue spaces have also been studied. 
In \cite{Rubin2014}, Rubin established weighted estimates for the $k$--plane transform as well as for its dual and the more general $j$--plane to $k$--plane transform acting on $L^p$ spaces with radial power weights. 
These results extend earlier weighted inequalities for the Radon transform obtained by Quinto \cite{Quinto1982} and Kumar and Ray \cite{Kumar2010}.

More precisely, let $1 \le p \le \infty$, and let $p'$ be the conjugate exponent to $p$. 
If 
\[
\nu = \mu - \frac{k}{p'}, 
\qquad 
\mu > k - \frac{d}{p},
\]
then the $k$--plane transform satisfies
\begin{align}\label{Rubin-estimate}
\|\, |\cdot|^\nu P f \,\|_{L^p(\cG_{d,k})}
\le C \,
\|\, |\cdot|^\mu f \,\|_{L^p(\R^d)},
\end{align}
(see \cite[Theorem 1.1]{Rubin2014} for an explicit expression of $C$.)

Related weighted $L^p$-norm inequalities appear in earlier and more recent work.
Solmon \cite{Solmon1979} proved that if $1 \le p < d/k$ and $0<\gamma<k$, then
\[
\bigl\|\, \langle \cdot \rangle^{\gamma-d}\, Pf \,\bigr\|_{L^1(\cG_{k,d})}
\le C\,
\bigl\|\, \langle \cdot \rangle^{k}\, f \,\bigr\|_{L^1(\R^d)},
\qquad 
\langle x \rangle = (1+|x|^2)^{1/2}.
\]
More recently, Flock \cite{Flock2025} obtained that if
$1 < p \le \frac{d+1}{k+1}$ and $\frac{d}{p}-\frac{d-k}{q}=k$, then
\[
\bigl\|\, \langle \cdot \rangle^{d}\, Pf \,\bigr\|_{L^q(\cG_{k,d})}
\le C\,
\bigl\|\, \langle \cdot \rangle^{\frac{d-k}{p-1}}\, f \,\bigr\|_{L^p(\R^d)}.
\]

Among the weighted results, Rubin's estimate \eqref{Rubin-estimate}
shows that the $k$--plane transform maps weighted $L^p$ spaces into
weighted $L^p$ spaces on the affine Grassmannian, with the relation
$\nu=\mu-\frac{k}{p'}$ precisely quantifying how integration over
$k$--dimensional planes reduces decay.
In contrast, the inequality of Flock allows a change of exponent
$L^p\to L^q$ consistent with the scaling relation
$\frac{d}{p}-\frac{d-k}{q}=k$, but requires stronger polynomial
weights on both $f$ and $Pf$.

\section{Sobolev Space Estimates}\label{sec:Sobolev-estimates}
While the Lebesgue estimates of the previous section describe the
integrability and decay of the $k$-plane integrals $Pf(\alpha,y)$,
Sobolev estimates presented in this section quantify the smoothing effect of the
$k$-plane transform.

\subsection{Sobolev Spaces on $\R^d$ and $\cG_{k,d}$}
Let $s\in\R$. The Sobolev norm of order $s$ of a function $f:\R^d\to\C$ is defined by
\begin{align}\label{sobolev-norm}
\|f\|_{H^s(\R^d)}
=
\left(
\int_{\R^d} (1+|\xi|^2)^s |\hat f(\xi)|^2 d\xi
\right)^{1/2},
\end{align}
where $\hat f$ denotes the Fourier transform of $f$.
The Sobolev space $H^s(\R^d)$ is defined as the completion of
$\mathcal S(\R^d)$ with respect to this norm.

As observed by Strichartz~\cite{Strichartz1986}, the Grassmannian bundle
$\cG_{k,d}$ does not admit a Riemannian or semi-Riemannian metric that is
invariant under the full Euclidean motion group. Consequently, there is no canonical invariant Laplace--Beltrami operator on $\cG_{k,d}$ and therefore no natural isotropic Sobolev scale. Instead, one defines Sobolev spaces using derivatives only in the fiber variables $y\in\alpha^\perp$.

For functions $u:\cG_{k,d}\to\C$ we define the Sobolev norm of order $s$ by
\begin{align}\label{eq:sobolev-norm-G}
\|u\|_{H^s(\cG_{k,d})}^2
=
\int_{G_{k,d}}\int_{\alpha^\perp}
(1+|\eta|^2)^s
|\widehat{u}(\alpha,\eta)|^2
\,d\eta\,d\alpha,
\end{align}
where $\widehat{u}(\alpha,\eta)$ denotes the Fourier transform of
$u(\alpha,y)$ in the fiber variable $y$.

The space $H^s(\cG_{k,d})$ is defined as the completion of the space of test functions with respect to this norm.

Thus, $H^s(\cG_{k,d})$ measures Sobolev regularity of order $s$ in the
fiber variables $y\in\alpha^\perp$, while the variable $\alpha$ is treated only in $L^2(G_{k,d})$.

\subsection{Sobolev Estimates for the $k$–Plane Transform}
Sobolev space estimates for the X-ray and Radon transforms have been
established in the literature (see \cite[Theorem 5.1]{Natterer}). The
next theorem gives the corresponding estimate for the $k$-plane transform.
In particular, it shows that the transform increases Sobolev regularity
by $k/2$ derivatives, consistent with its microlocal description as
a Fourier integral operator of order $-k/2$ \cite{Chihara2022}.

\begin{theorem}\label{thm:sobolev-estimate}
Let $s\in\R$ and $f\in H^s(\R^d)$ with $\operatorname{supp}(f)\subseteq\overline{\Omega}$ 
for some bounded open set $\Omega\subset\R^d$.
Then there exist constants $c_{s,d,k},C_{s,d,k}>0$ such that
\begin{align}\label{StabilitySobolev}
c_{s,d,k}\|f\|_{H^s(\R^d)}
\le
\|Pf\|_{H^{s+k/2}(\cG_{k,d})}
\le
C_{s,d,k}\|f\|_{H^s(\R^d)}.
\end{align}
\end{theorem}

\begin{proof}
By the above definitions and the Fourier-slice theorem \ref{t:FST}, we have
\begin{align*}
\|Pf\|_{H^{s+k/2}(\cG_{k,d})}^2 
&= \int_{G_{k,d}} \int_{\alpha^\perp} (1+|\eta|^2)^{s+k/2}|\widehat{Pf}(\alpha,\eta)|^2 \, d\eta \, d\alpha \\
&= (2\pi)^{k}\int_{G_{k,d}} \int_{\alpha^\perp} (1+|\eta|^2)^{s+k/2}|\hat{f}(\eta)|^2 \, d\eta \, d\alpha \\
&= (2\pi)^{k}|G_{k,d-1}|\int_{\R^d} |\xi|^{-k} (1+|\xi|^2)^{s+k/2}|\hat{f}(\xi)|^2 \, d\xi,
\end{align*}
where the last equality is implied by Lemma \ref{generalized_polar}. 
Since $(1+|\xi|^2)^{k/2} \geq |\xi|^k$ for all $\xi \in \R^d$, we obtain the left-hand side inequality with 
$c_{s,d,k} = \sqrt{(2\pi)^{k}|G_{k,d-1}|}$.

For the right-hand side inequality, we split the last integral into two parts:
\begin{align*}
\int_{\R^d} |\xi|^{-k} (1+|\xi|^2)^{s+k/2}|\hat{f}(\xi)|^2 \, d\xi 
&= \int_{|\xi| \geq 1} |\xi|^{-k} (1+|\xi|^2)^{s+k/2}|\hat{f}(\xi)|^2 \, d\xi \\
&\quad + \int_{|\xi| \leq 1} |\xi|^{-k} (1+|\xi|^2)^{s+k/2}|\hat{f}(\xi)|^2 \, d\xi.
\end{align*}

For $|\xi| \geq 1$, we have $1+|\xi|^2 \le 2|\xi|^2$, and thus
\[
\int_{|\xi| \geq 1} |\xi|^{-k} (1+|\xi|^2)^{s+k/2}|\hat{f}(\xi)|^2 \, d\xi
\leq 2^{k/2} \int_{|\xi| \geq 1} (1+|\xi|^2)^s|\hat{f}(\xi)|^2 \, d\xi
\leq 2^{k/2}\|f\|^2_{H^s}.
\]

For $|\xi| \leq 1$, we estimate
\[
\int_{|\xi| \leq 1} |\xi|^{-k} (1+|\xi|^2)^{s+k/2}|\hat{f}(\xi)|^2 \, d\xi
\leq 
\left(\int_{|\xi| \leq 1} |\xi|^{-k} (1+|\xi|^2)^{s+k/2} \, d\xi \right)\sup_{|\xi| \leq 1}|\hat{f}(\xi)|^2.
\]

Writing $\xi = r\omega$ with $r=|\xi|$ and $\omega\in\mathbb S^{d-1}$ gives
\begin{align*}
\int_{|\xi|\le 1} &|\xi|^{-k}(1+|\xi|^2)^{s+k/2}\,d\xi
=|\mathbb S^{d-1}|
\int_0^1 r^{d-1-k}(1+r^2)^{s+k/2}\,dr\\
&\le \max\{1, 2^{s+k/2}\}|\mathbb S^{d-1}| 
\int_0^1 r^{d-1-k}\,dr
= \max\{1, 2^{s+k/2}\}| \frac{|\mathbb S^{d-1}|}{d-k},
\end{align*}
which is finite as $k<d$.

It remains to estimate $\sup_{|\xi|\le 1}|\hat{f}(\xi)|^2$, and this is where the compact support of $f$ is needed.
Let $\chi \in C_c^{\infty}(\R^d)$ satisfy $\chi = 1$ on $\Omega$, and define $\chi_{\xi}(x) = e^{-ix\cdot \xi} \chi(x)$.
Since $\operatorname{supp}(f)\subseteq \Omega$, we have
\[
\hat{f}(\xi) 
= (2\pi)^{-d/2}\int_{\R^d} f(x)e^{-ix\cdot\xi}\,dx
= (2\pi)^{-d/2}\int_{\R^d} f(x)\chi_\xi(x)\,dx.
\]
Set $g_\xi(x) := (2\pi)^{-d/2}\chi_\xi(x)$. Then $\hat f(\xi)=\langle f,g_\xi\rangle$.

By Sobolev duality $(H^s(\R^d))^* = H^{-s}(\R^d)$ (see, e.g., \cite[Ch.~3]{Adams}), we have the estimate
\[
|\hat f(\xi)| = |\langle f,g_\xi\rangle|
\le \|f\|_{H^s}\,\|g_\xi\|_{H^{-s}}
= (2\pi)^{-d/2}\|f\|_{H^s}\,\|\chi_\xi\|_{H^{-s}}.
\]
Therefore it suffices to show that $\sup_{|\xi|\le 1}\|\chi_\xi\|_{H^{-s}}<\infty$.

Using $\widehat{\chi_\xi}(\zeta)=\hat\chi(\zeta+\xi)$, we compute
\[
\|\chi_\xi\|_{H^{-s}}^2
= \int_{\R^d} (1+|\zeta|^2)^{-s}\,|\widehat{\chi_\xi}(\zeta)|^2\,d\zeta
= \int_{\R^d} (1+|\zeta|^2)^{-s}\,|\hat\chi(\zeta+\xi)|^2\,d\zeta.
\]
Changing variables $z=\zeta+\xi$ gives
\[
\|\chi_\xi\|_{H^{-s}}^2
= \int_{\R^d} (1+|z-\xi|^2)^{-s}\,|\hat\chi(z)|^2\,dz.
\]

Assume $|\xi|\le 1$. From $|z-\xi|\le |z|+1$ and $|z|\le |z-\xi|+1$ we obtain
\[
\frac{1}{3}(1+|z|^2) \le 1+|z-\xi|^2 \le 3(1+|z|^2).
\]

Thus, for all $s\in\R$,
\[
(1+|z-\xi|^2)^{-s} \le 3^{|s|}(1+|z|^2)^{-s},
\qquad |\xi|\le 1.
\]

Hence
\[
\|\chi_\xi\|_{H^{-s}}^2
\le 3^{|s|}\int_{\R^d}(1+|z|^2)^{-s}\,|\hat\chi(z)|^2\,dz.
\]
Since $\chi\in C_c^\infty(\R^d)$, its Fourier transform $\hat\chi$ belongs to the Schwartz space, and therefore for all $s\in\R$,
\[
\int_{\R^d}(1+|z|^2)^{-s}\,|\hat\chi(z)|^2\,dz <\infty.
\]
Thus, there exists a constant $C_{s,d,\chi}>0$ such that
\[
\sup_{|\xi|\le 1}\|\chi_\xi\|_{H^{-s}}^2 \le C_{s,d,\chi}.
\]
Combining the above inequalities yields
\[
\sup_{|\xi|\le 1}|\hat f(\xi)|^2
\le (2\pi)^{-d}\,C_{s,d,\chi}\,\|f\|_{H^s}^2.
\]
This gives the desired bound for the low-frequency region $|\xi|\le 1$, and hence the right-hand side inequality in \eqref{StabilitySobolev}.
We note that for $s<0$ the duality pairing $\langle f,g_\xi\rangle$ is understood in the sense of distributions.
\end{proof}

\subsection{Weighted Sobolev Estimates}\label{sec:weighted-sobolev estimates}

Let $1\le p<\infty$, $s\in\R$, and $t>-d/p$. We define the weighted
Sobolev space on $\R^d$ by
\begin{align}\label{weighted-sobolev}
\|f\|_{H^{s,p}_t(\R^d)}^p
=
\int_{\R^d}
\left(
|\xi|^{t}(1+|\xi|^2)^{\frac{s-t}{2}}
|\hat f(\xi)|
\right)^p
\,d\xi .
\end{align}

Similarly, for functions on the affine Grassmannian $\cG_{k,d}$ we define
\begin{align}\label{weighted-sobolev-G}
\|u\|_{H^{s,p}_t(\cG_{k,d})}^p
=
\frac{(2\pi)^{-pk/2}}{|G_{k,d-1}|}
\int_{G_{k,d}}
\int_{\alpha^\perp}
\left(
|\eta|^{t}(1+|\eta|^2)^{\frac{s-t}{2}}
|\widehat u(\alpha,\eta)|
\right)^p
\,d\eta\,d\alpha .
\end{align}

When $p=2$, these spaces reduce to classical Sobolev spaces: if $t=0$
we recover the usual Sobolev space ($H^{s,2}_0 = H^{s}$),
while $t=s$ corresponds to the homogeneous Sobolev space
($H^{s,2}_s = \dot H^{s}$). Moreover, for compactly supported
functions the norm \eqref{weighted-sobolev} is equivalent to the
classical Sobolev norm \eqref{sobolev-norm} for any $s\in\R$ and
$|t|<d/2$, see \cite[Lemma~2.3]{Sharafutdinov2017tensor}.

The Fourier weight
\[
w_t(\xi):=|\xi|^{t}(1+|\xi|^2)^{-t/2}
\]
behaves like $|\xi|^{t}$ as $|\xi|\to0$ and tends to a constant as
$|\xi|\to\infty$. Consequently, the weight alters only the contribution of low frequencies to the Sobolev norm: for $t>0$ low frequencies are suppressed, whereas for $t<0$ they are amplified.

Identities relating norms of this type to integral transforms have
appeared previously in several settings. For the hyperplane Radon
transform ($k=d-1$), Reshetnyak \cite{GGV1966} established such an
identity in the $L^2$ setting ($p=2$, $s=0$, and $t=0$). Sharafutdinov extended this result to arbitrary $s\in\R$ and later to weighted norms with $t>-d/2$
\cite{Sharafutdinov2017tensor,Sharafutdinov2021radon}. He also treated the case
of the X-ray transform ($k=1$) under the same conditions \cite{Sharafutdinov2021x-ray}.
More recently, Kindermann and Hubmer \cite{Kindermann2025} proved the
corresponding identity for all $1\le p<\infty$ in the hyperplane Radon
transform case, allowing arbitrary $s\in\R$ and weights $t>-d/p$.

The following theorem establishes the corresponding identity for the
general $k$--plane transform.

\begin{theorem}\label{thm:P-isometry}
Let $d \ge 2$ and $1 \le k \le d-1$. Suppose that $1\le p<\infty$, $s\in\R$, and $t>-d/p$. Then
\begin{align}\label{P-isometry}
\|Pf\|_{H^{s+k/p,p}_{t+k/p}(\cG_{k,d})}
=
\|f\|_{H^{s,p}_t(\R^d)} .
\end{align}
\end{theorem}

\begin{proof}
By definition \eqref{weighted-sobolev-G} and the Fourier slice theorem
(Theorem~\ref{t:FST}) we obtain
\begin{align*}
\|Pf\|_{H^{s+k/p,p}_{t+k/p}(\cG_{k,d})}^p
&=
\frac{(2\pi)^{-pk/2}}{|G_{k,d-1}|}
\int_{G_{k,d}}
\int_{\alpha^\perp}
\left(
|\eta|^{t+k/p}(1+|\eta|^2)^{\frac{s-t}{2}}
|\widehat{Pf}(\alpha,\eta)|
\right)^p
d\eta\,d\alpha
\\
&=
\frac{1}{|G_{k,d-1}|}
\int_{G_{k,d}}
\int_{\alpha^\perp}
\left(
|\eta|^{t+k/p}(1+|\eta|^2)^{\frac{s-t}{2}}
|\hat f(\eta)|
\right)^p
d\eta\,d\alpha .
\end{align*}

Applying lemma \ref{generalized_polar} on $\R^d$ yields
\begin{align*}
\|Pf\|_{H^{s+k/p,p}_{t+k/p}(\cG_{k,d})}^p
&=
\int_{\R^d}
\left(
|\xi|^{t}(1+|\xi|^2)^{\frac{s-t}{2}}
|\hat f(\xi)|
\right)^p
d\xi
=
\|f\|_{H^{s,p}_t(\R^d)}^p ,
\end{align*}
which proves \eqref{P-isometry}.
\end{proof}

\section{Besov and Triebel--Lizorkin Space Estimates}\label{sec:BL-TL-estimates}
In this section, we study the mapping properties of the $k$-plane transform in Besov and Triebel--Lizorkin spaces. These spaces form a refined scale of smoothness spaces extending the classical Sobolev scale. While Sobolev spaces measure smoothness in a global averaged sense, Besov and Triebel–Lizorkin spaces provide a finer description of regularity by decomposing functions into components localized in different frequency bands via Littlewood–Paley decompositions. The include many classical function spaces as special cases, including Sobolev and Bessel--potential spaces, Hölder--Zygmund spaces, local Hardy spaces, and functions of bounded mean oscillation. For a thorough study of these spaces we refer to  \cite{Triebel1983,Triebel1992,Grafakos2014classical,Grafakos2014modern,Hytonen2023}.

\subsection{Besov and Triebel--Lizorkin Spaces on $\R^d$ and $\cG_{k,d}$}
We begin by introducing a smooth dyadic resolution of unity
$\{\varphi_j\}_{j\in\mathbb N_0}$.

Let $\varphi \in C^\infty(\mathbb{R}^d)$ satisfy
\[
\operatorname{supp}\varphi \subset \{\xi \in \R^d: |\xi| \le 2\},
\qquad
\varphi(\xi)=1 \quad \text{if } |\xi|\le 1.
\]

For $j \in \N$, define
\[
\varphi_j(\xi):=\varphi(2^{-j}\xi)-\varphi(2^{-j+1}\xi).
\]

Then
\[
\operatorname{supp}\varphi_j
\subset
\{2^{j-1}\le |\xi|\le 2^{j+1}\},
\quad j\in\N,
\]
and, with $\varphi_0=\varphi$,
\[
\sum_{j=0}^{\infty}\varphi_j(\xi)=1,
\qquad \xi\in\R^d.
\]

For $f\in\mathcal S'(\R^d)$ we define the associated
Littlewood--Paley projections by
\begin{align}\label{LPproj_f}
M_j f :=(\varphi_j \widehat{f})^\vee,
\qquad j\in\N_0 .
\end{align}

Each operator $M_j$ localizes a function to frequencies of size
$\sim 2^j$, and we have the decomposition
\[
f=\sum_{j=0}^\infty M_jf,
\]
with convergence in $\cS'$.

We now introduce the (inhomogeneous) Besov and Triebel--Lizorkin spaces.
\begin{definition}
Let $s\in\R$ and $0<r\le\infty$.
\begin{enumerate}
\item[(i)]
For $0<p\le\infty$, the Besov space $B^s_{p,r}(\R^d)$ consists of all
$f\in\mathcal S'(\R^d)$ such that
\begin{align}\label{BL-norm}
\|f\|_{B^s_{p,r}}
:=
\Big(
\sum_{j=0}^\infty
\big(2^{js}\|M_j f\|_{L^p}\big)^r
\Big)^{1/r}
<\infty .
\end{align}

\item[(ii)]
For $0<p<\infty$, the Triebel--Lizorkin space $F^s_{p,r}(\R^d)$ consists
of all $f\in\mathcal S'(\R^d)$ such that
\begin{align}\label{TL-norm}
\|f\|_{F^s_{p,r}}
:=
\left\|
\left(
\sum_{j=0}^\infty (2^{js}|M_j f|)^r
\right)^{1/r}
\right\|_{L^p}
<\infty .
\end{align}
\end{enumerate}
If $r=\infty$, the sums are replaced by the supremum.
\end{definition}

Different choices of the dyadic resolution of unity $\{\varphi_j\}$
lead to equivalent norms, so the spaces $B^s_{p,r}$ and $F^s_{p,r}$
are independent of $\{\varphi_j\}$ (see \cite{Triebel1983}, Section 2.3.2).

\medskip

To define Besov and Triebel--Lizorkin spaces on $\cG_{k,d}$ we apply the
Littlewood--Paley decomposition in the fiber frequency variable
$\eta\in\alpha^\perp$.

Let $\{\varphi_j\}_{j\in\N_0}$ be the dyadic partition introduced above.
For $u\in\cS'(\cG_{k,d})$ we define the associated Littlewood--Paley
projections by
\begin{align}\label{LPproj_u}
\cM_j u(\alpha,y)
:=
(\varphi_j(\cdot)\widehat u(\alpha,\cdot))^\vee(y),
\qquad j\in\N_0 .
\end{align}
Then
\[
u=\sum_{j=0}^\infty \cM_j u,
\]
with convergence in $\cS'(\cG_{k,d})$.

\begin{definition}\label{BL-TL-spaceG}
Let $s\in\R$, $0<q\le\infty$, and $0<r\le\infty$.

\begin{enumerate}

\item[(i)]
For $0<t\le\infty$, the anisotropic Besov space
$B^s_{(q,t),r}(\cG_{k,d})$ consists of all
$u\in\cS'(\cG_{k,d})$ such that
\begin{align}\label{aniBL-normG}
\|u\|_{B^s_{(q,t),r}(\cG_{k,d})}
:=
\Big(
\sum_{j=0}^\infty
\big(
2^{js}\|\cM_j u\|_{L^q_\alpha(L^t_y)}
\big)^r
\Big)^{1/r}
<\infty .
\end{align}

\item[(ii)]
For $0<t<\infty$, the anisotropic Triebel--Lizorkin space
$F^s_{(q,t),r}(\cG_{k,d})$ consists of all
$u\in\cS'(\cG_{k,d})$ such that
\begin{align}\label{aniTL-normG}
\|u\|_{F^s_{(q,t),r}(\cG_{k,d})}
:=
\left\|
\left(
\sum_{j=0}^\infty (2^{js}|\cM_j u|)^r
\right)^{1/r}
\right\|_{L^q_\alpha(L^t_y)}
<\infty .
\end{align}
\end{enumerate}
If $r=\infty$, the sums are replaced by the supremum.
\end{definition}

These definitions are independent (up to equivalence of norms) of the particular choice of dyadic resolution of unity
$\{\varphi_j\}_{j=0}^\infty$, since the Littlewood--Paley
decomposition acts only in the Euclidean fiber variable
$\eta\in\R^{d-k}$ and the resulting estimates can be applied
fiberwise and integrated over $G_{k,d}$.

For $t=q$, we write
\[
B_{q,r}^s(\cG_{k,d})
:=
B_{(q,q),r}^s(\cG_{k,d}),
\qquad
F_{q,r}^s(\cG_{k,d})
:=
F_{(q,q),r}^s(\cG_{k,d}).
\]

In the case $q=r=2$, the Besov and Triebel--Lizorkin spaces
coincide with the fiber Sobolev space introduced in
Section~\ref{sec:Sobolev-estimates}.

\begin{proposition}
For every $s\in\R$,
\[
F_{2,2}^s(\cG_{k,d})
=
B_{2,2}^s(\cG_{k,d})
=
H^s(\cG_{k,d}),
\]
with equivalent norms.
\end{proposition}

\begin{proof}
Let $u\in\cS(\cG_{k,d})$. By definition and Fubini's theorem,
\[
\|u\|_{F_{2,2}^s(\cG_{k,d})}^2
=
\int_{G_{k,d}}
\sum_{j=0}^\infty 2^{2js}\int_{\alpha^\perp} |\cM_j u(\alpha,y)|^2
\,dy\,d\alpha .
\]

For each fixed $\alpha$, Plancherel's theorem in the fiber variable gives
\[
\int_{\alpha^\perp} |\cM_j u(\alpha,y)|^2\,dy
=
\int_{\alpha^\perp}
|\varphi_j(\eta)\widehat u(\alpha,\eta)|^2
\,d\eta .
\]

Hence, again by Fubini's theorem,
\[
\|u\|_{F_{2,2}^s(\cG_{k,d})}^2
=
\int_{G_{k,d}}\int_{\alpha^\perp}
\Big(\sum_{j=0}^\infty
2^{2js}|\varphi_j(\eta)|^2\Big)
|\widehat u(\alpha,\eta)|^2
\,d\eta\,d\alpha .
\]

By definition of the dyadic resolution of unity $\{\varphi_j\}$, we have
\[
c_s(1+|\eta|^2)^s
\le
\sum_{j=0}^\infty 2^{2js}|\varphi_j(\eta)|^2
\le
C_s(1+|\eta|^2)^s,
\]
for some $c_s, C_s >0$.
Thus we obtain
\[
c_s \|u\|_{H^s(\cG_{k,d})}^2 \le \|u\|_{F_{2,2}^s(\cG_{k,d})}^2
\le c_s \|u\|_{H^s(\cG_{k,d})}^2.
\]
The same computation, with the sum over $j$ taken outside the $L^2$--norm, shows that
$\|u\|_{B_{2,2}^s(\cG_{k,d})}$ is also equivalent to $\|u\|_{H^s(\cG_{k,d})}$. Density of
$\cS(\cG_{k,d})$ in all three spaces and standard completion arguments then yield the result. 
\end{proof}

\subsection{Besov and Triebel--Lizorkin Space Estimates}
The next result combines the mixed-norm Lebesgue estimates for $P$ (cf. section \ref{sec:Lp-estimates}) with Littlewood--Paley decompositions. The key ingredients are that $P$ intertwines Euclidean Fourier multipliers with their fiberwise counterparts on $\cG_{k,d}$ (Proposition~\ref{prop:intertwine-kplane}) and the vector valued inequality for $P$ (Proposition~\ref{prop:vector-valued inequality}).

\begin{theorem}[Anisotropic Besov and Triebel--Lizorkin estimates]
\label{thm:BL-TL-estimate}
Let $d\ge2$ and $1\le k\le d-1$. Suppose the exponents $p,q,t$
satisfy
\begin{align}\label{eq:aniso-extra}
1\le p\le \frac{d+1}{k+1},
\qquad
\text{\emph{or}}\quad p\le2\ \text{ \emph{and} }\ p<\frac{d}{k},
\end{align}
and
\begin{align}\label{eq:aniso-range}
\frac{d}{p}-\frac{d-k}{t}=k,
\qquad
q\le (d-k)p'.
\end{align}
Then, for any $1\le r \le \infty,\; s \in \R$, there exists a constant $C>0$ (depending on $k,d,p,q,t$) such that
\begin{align}\label{eq:aniBL-bound}
\|Pf\|_{B^s_{(q,t),r}(\cG_{k,d})}\ \le\ C\,\|f\|_{B^s_{p,r}(\R^d)},
\end{align}
and,
\begin{align}\label{eq:aniTL-bound}
\|Pf\|_{F^{s}_{(q,t),r}(\cG_{k,d})}
\ \le\ C\,\|f\|_{F^{s}_{p,r}(\R^d)}.
\end{align}
Thus, the $k$-plane transform $P$ is bounded as a map from $B^{s}_{p,r}(\R^d) \to B^{s}_{(q,t),r}(\cG_{k,d})$ and $F^{s}_{p,r}(\R^d) \to F^{s}_{(q,t),r}(\cG_{k,d})$.
\end{theorem}
\begin{proof}
Let $d,k,p,q,t,r,s$ be as in the theorem. For the admissible parameters, $\mathcal S(\R^d)$ is dense in $B^s_{p,r}(\R^d)$ and $F^s_{p,r}(\R^d)$ (see \cite[Theorem~2.3.3]{Triebel1983}), so it suffices to prove \eqref{eq:aniBL-bound} and \eqref{eq:aniTL-bound} for $f\in\mathcal S(\R^d)$.

Let $f\in\mathcal S(\R^d)$. By the intertwining property
(Proposition~\ref{prop:intertwine-kplane}), for each $j\in\N_0$,
\begin{align}\label{eq:thm-intertwine}
\cM_j (Pf) \,=\, P (M_j f).
\end{align}

We first prove the Besov estimate \eqref{eq:aniBL-bound}. Using \eqref{eq:thm-intertwine} and the mixed norm estimate of Theorem~\ref{thm:christ-mixed}, we obtain for each $j\in\N_0$,
\[
\|\cM_j(Pf)\|_{L_\alpha^q(L_y^t)}
=
\|P(M_j f)\|_{L_\alpha^q(L_y^t)}
\le
C\,\|M_j f\|_{L^p(\R^d)}.
\]
Multiplying by $2^{js}$ and taking the $\ell^r$-norm over $j$ gives
\begin{align*}
\|Pf\|_{B^{s}_{(q,t),r}(\cG_{k,d})}
&=
\Big(\sum_{j=0}^\infty\big(2^{js}\|\cM_j(Pf)\|_{L_\alpha^q(L_y^t)}\big)^r\Big)^{1/r}\\
&\le
C\Big(\sum_{j=0}^\infty\big(2^{js}\|M_j f\|_{L^p(\R^d)}\big)^r\Big)^{1/r}
=
C\,\|f\|_{B^{s}_{p,r}(\R^d)},
\end{align*}
which proves \eqref{eq:aniBL-bound}.

Next, we prove the Triebel--Lizorkin bound \eqref{eq:aniTL-bound}. By \eqref{aniTL-normG} and \eqref{eq:thm-intertwine},
\begin{align*}
\|Pf\|_{F^s_{(q,t),r}(\cG_{k,d})}
&=\Big\|\Big(\sum_{j=0}^\infty\big(2^{js}|\cM_j(Pf)|\big)^r\Big)^{1/r}\Big\|_{L_\alpha^q(L_y^t)}\\
&=\Big\|\Big(\sum_{j=0}^\infty\big(2^{js}|P(M_jf)|\big)^r\Big)^{1/r}\Big\|_{L_\alpha^q(L_y^t)}.
\end{align*}

Since $P$ is linear, $2^{js} P(M_j f) = P(2^{js} M_j f)$.
Applying Proposition~\ref{prop:vector-valued inequality} with
$h_j=2^{js}M_jf$ yields
\begin{align}\label{use vector-valued inequality}
\Big(\sum_{j=0}^\infty|P(2^{js}M_j f)|^r\Big)^{1/r}
\le
 P\Big(\Big(\sum_{j=0}^\infty|2^{js}M_j f|^r\Big)^{1/r}\Big) \quad \text{pointwise}.
\end{align}
Define
\[
F(x):=\Big(\sum_{j=0}^\infty\big(2^{js}|M_j f(x)|\big)^r\Big)^{1/r}.
\]
Taking the mixed norm $L^q_\alpha(L^t_y)$ in \eqref{use vector-valued inequality} yields
\begin{align}\label{eq:thm-reduce-to-LpLq}
\|Pf\|_{F^s_{(q,t),r}(\cG_{k,d})}\leq  \|P(F)\|_{L_\alpha^q(L_y^t)}.
\end{align}
Under \eqref{eq:aniso-range}--\eqref{eq:aniso-extra}, Theorem~\ref{thm:christ-mixed} implies
\[
\|P(F)\|_{L_\alpha^q(L_y^t)}\le C\,\|F\|_{L^p(\R^d)} = C\|f\|_{F^s_{p,r}(\R^d)},
\]
where the last identity is exactly \eqref{TL-norm}. Combining with \eqref{eq:thm-reduce-to-LpLq} proves \eqref{eq:aniTL-bound}.
\end{proof}

By taking $t=q$ in Theorem~\ref{thm:BL-TL-estimate}, we obtain an isotropic version. Indeed, substituting $t=q$ into \eqref{eq:aniso-range} gives $q=(k+1)p$.
\begin{corollary}[Isotropic Besov and Triebel--Lizorkin estimates]\label{thm:TL-Besov-upper-P}
For $d\ge2$ and $1\le k\le d-1$, let
\[
1 \le p \le \frac{d+1}{k+1},\qquad q=(k+1)p,\qquad 1\le r \le \infty,\qquad s\in\R.
\]
Then, there exists a constant $C >0$ (depending on $k,d,p,q$) such that
\begin{align}\label{eq:thm-Besov}
\|Pf\|_{B^{s}_{q,r}(\cG_{k,d})}\ \le\ C\,\|f\|_{B^{s}_{p,r}(\R^d)}.
\end{align}
and
\begin{align}\label{eq:thm-TL}
\|Pf\|_{F^{s}_{q,r}(\cG_{k,d})}\ \le\ C\,\|f\|_{F^{s}_{p,r}(\R^d)}.
\end{align}
\end{corollary}

\section{Conclusions}\label{sec:conclusion}

In this paper we studied mapping properties of the $k$-plane transform in several classical function spaces. 
First, we established Sobolev stability estimates for compactly supported functions, extending Natterer's estimates for the X-ray and Radon transforms to the general $k$-plane transform. This result shows that the $k$-plane transform improves Sobolev regularity by $k/2$ derivatives, in agreement with its microlocal description as a Fourier integral operator of order $-k/2$.
We also proved an isometry identity for weighted Sobolev spaces, extending earlier results obtained for the Radon and X-ray transforms by Reshetnyak, Sharafutdinov, and Kindermann--Hubmer to the general $k$-plane transform.

We then established boundedness of the $k$-plane transform in Besov and Triebel--Lizorkin spaces. In contrast to the Sobolev estimates, these results do not require compact support, but the smoothness parameter remains unchanged. Here, the analysis relies on the known $L^p$ estimates for the $k$-plane transform together with its intertwining relation with Fourier multiplier operators and a vector-valued inequality. These properties allow Lebesgue space estimates to be transferred to more refined smoothness scales.

Several directions remain for future investigation. It would be natural to establish corresponding lower bounds and stability estimates for the $k$-plane transform in Besov and Triebel--Lizorkin spaces. Another natural question is whether a shift in the smoothness scale occurs in these function space settings, as in the Sobolev case. Finally, it would be of interest to explore whether the estimates obtained here can be used in the analysis of inversion and reconstruction algorithms for the $k$-plane transform.

\section*{Acknowledgements}
This work was supported in part by NSF DMS grant 2206279. An AI-based language tool was used to assist in editing the manuscript for spelling, grammar, and stylistic improvements.

\bibliographystyle{siam}
\bibliography{references}

\end{document}